\documentclass[12pt, reqno]{amsart}
\usepackage{amsmath, amsthm, amscd, amsfonts, amssymb, graphicx, color}
\usepackage[bookmarksnumbered, colorlinks, plainpages]{hyperref}
\hypersetup{colorlinks=true,linkcolor=red, anchorcolor=green, citecolor=cyan, urlcolor=red, filecolor=magenta, pdftoolbar=true}
\input{mathrsfs.sty}

\newtheorem{theorem}{Theorem}[section]
\newtheorem{lemma}[theorem]{Lemma}
\newtheorem{proposition}[theorem]{Proposition}
\newtheorem{cor}[theorem]{Corollary}
\theoremstyle{definition}
\newtheorem{definition}[theorem]{Definition}

\theoremstyle{remark}

\numberwithin{equation}{section}
\usepackage{physics}
\begin{document}

\title [Euclidean operator radius inequalities of $d$-tuple operators] {Euclidean operator radius inequalities of $d$-tuple operators and operator matrices}

\author[S. Jana, P. Bhunia, K. Paul] {Suvendu Jana, Pintu Bhunia, Kallol Paul*}

\address{(Jana) Department of Mathematics, Mahishadal Girls' College, Purba Medinipur 721628, West Bengal, India} \email{janasuva8@gmail.com}

\address{(Bhunia) Department of Mathematics, Indian Institute of Science, Bengaluru-560012, Karnataka, India} \email{pintubhunia5206@gmail.com}

\address{(Paul) Department of Mathematics, Jadavpur University, Kolkata 700032, West Bengal, India} \email{kalloldada@gmail.com}


\thanks{* Corresponding author. \\Dr. Pintu Bhunia would like to thank SERB, Govt. of India for the financial support in the form of National Post Doctoral Fellowship (N-PDF, File No. PDF/2022/000325) under the mentorship of Prof. Apoorva Khare}

\renewcommand{\subjclassname}{\textup{2020} Mathematics Subject Classification}

\maketitle

\begin{abstract}
	In this paper, we develop several Euclidean operator radius inequalities of $d$-tuple operators, as well as the sum and the product of $d$-tuple operators. Also, we obtain a power inequality for the Euclidean operator radius.
	Further, we develop Euclidean operator radius inequalities of $2\times 2$ operator matrices whose entries are $d$-tuple operators.
\end{abstract}


\section{Introduction}
\noindent 
Let $\mathscr{H}$ be a complex Hilbert space with usual inner product $\langle \cdot,\cdot \rangle $, and  $\|\cdot\|$ be the norm induced by the inner product. Let $ \mathbb{B}(\mathscr{H})$ denote the $C^*$-algebra of all bounded linear operators on $\mathscr{H}.$ For any $T\in\mathbb{B}(\mathscr{H})$, the numerical range of $T$ is given by $ W(T)=\{ \langle Tx,x\rangle: x\in\mathscr{H}, \|x\|=1\}$ and the numerical radius of $T$, denoted by $ w(T)$, is defined as $ w(T)=\sup\hspace{0.2cm} \{ |\lambda| : \lambda\in W(T) \}.$ It is well known that $w(\cdot)$ defines a norm on $\mathbb{B}(\mathscr{H})$, and it satisfies the inequality $\frac12 \|T\|\leq w(T) \leq \|T\|$.
For further readings on the numerical range and the numerical radius inequalities, see the books \cite{BDMP, GW, GUS}.
Now, let $\mathbb{B}^d(\mathscr{H})= \mathbb{B}(\mathscr{H}) \times \mathbb{B}(\mathscr{H}) \times \ldots \times \mathbb{B}(\mathscr{H})$ ($d$ times) and   let $\mathbf{T}=(T_1,T_2,\ldots,T_d)\in  \mathbb{B}^d(\mathscr{H})$ be a $d$-tuple operator. The joint numerical range, joint numerical radius, joint Crawford number and joint operator norm of $\mathbf{T}$ are defined respectively as follows:
\begin{eqnarray*}
	&&  JtW(\mathbf{T})= \{ ( \langle T_1 x,x\rangle,\langle T_2  x,x\rangle,\ldots,\langle T_d x,x\rangle):  x\in\mathscr{H}, \|x\|=1 \},\\
	&&w_e (\mathbf{T})=\sup \left\lbrace \left( \sum_{k=1}^{d} |\langle T_k x,x\rangle|^2 \right)^\frac{1}{2}:  x\in\mathscr{H}, \|x\|=1\right\rbrace,\\
	&&c_e(\mathbf{T})=\inf \left\lbrace \left( \sum_{k=1}^{d} |\langle T_k x,x\rangle|^2 \right)^\frac{1}{2}:  x\in\mathscr{H}, \|x\|=1\right\rbrace,\\
	&&\|\mathbf{T}\|=\sup \left\lbrace \left( \sum_{k=1}^{d} \|T_kx\|^2 \right)^\frac12:x\in\mathscr{H}, \|x\|=1\right\rbrace.
\end{eqnarray*}
Note that $w_e (\mathbf{T})$ is also known as the Euclidean operator radius of $d$-tuple operator $\mathbf{T}.$ 
As pointed out in \cite{P}, $w_e(\cdot)$ is a norm on $\mathbb{B}^d(\mathscr{H})$ and satisfies the following inequality: \begin{eqnarray}
\frac{1}{2\sqrt{d}}\|\sum_{k=1}^{d} T_k^*T_k\|^{\frac12}\leq w_e(\mathbf{T})\leq\|\sum_{k=1}^{d} T_k^*T_k\|^{\frac12}.\label{T1}\end{eqnarray}
 Here the constant $\frac{1}{2\sqrt{d}}$ and $1$ are best possible.
For the latest and recent improvements of (\ref{T1}) the reader can see \cite{SD,SPK,MSS} and the references therein.
Next, we recall the following definitions of commuting $d$-tuple operator.
  \begin{definition} Let $\mathbf{T}=(T_1,T_2,\ldots,T_d)\in \mathbb{B}^d(\mathscr{H})$ be a $d$-tuple operator. Then $\mathbf{T}$ is said to be commuting if $T_iT_j=T_jT_i$ for all $ i,j=1,2,\ldots,d.$ \end{definition} 
 \begin{definition} \cite{C}  Let $\mathbf{T}=(T_1,T_2,\ldots,T_d)\in \mathbb{B}^d(\mathscr{H})$ be a $d$-tuple operator. Then $\mathbf{T}$  is said to be joint normal (or simply normal) if $\mathbf{T}$ is commuting and each $T_i$ is normal.
 \end{definition}
For $d$-tuple operators $ \mathbf{S}=(S_1,S_2,\ldots,S_d)$, $\mathbf{T}=(T_1,T_2,\ldots,T_d)\in\mathbb{B}^d(\mathscr{H})$, we write $\mathbf{ST}=(S_1T_1,S_2T_2,\ldots,S_dT_d)$, $\mathbf{S}+\mathbf{T}=(S_1+T_1,S_2+T_2,\ldots,S_d+T_d)$  
and $\alpha \mathbf{T}=(\alpha T_1,\alpha T_2, \ldots,\alpha T_d)$ for any scalar $\alpha \in\mathbb{C}$. 
Also, for $\mathbf{X}=(X_1,X_2, \ldots,X_d)$, $\mathbf{Y}=(Y_1,Y_2, \ldots,Y_d)$, $\mathbf{Z}=(Z_1,Z_2, \ldots,Z_d)$, $\mathbf{W}=(W_1,W_2, \ldots,W_d)\in\mathbb{B}^d(\mathscr{H})$,  the $2\times 2$ operator matrix, whose entries are $d$-tuple operators $\mathbf{X}, \mathbf{Y}, \mathbf{Z}, \mathbf{W},$ is defined as 
  $$\begin{bmatrix}
 \mathbf{X} & \mathbf{Y} \\
 \mathbf{Z} &  \mathbf{W} 
\end{bmatrix}=\left(\begin{bmatrix}
X_1 & Y_1 \\
Z_1 & W_1
\end{bmatrix} , \begin{bmatrix}
X_2 & Y_2 \\
Z_2 & W_2
\end{bmatrix}, \ldots,\begin{bmatrix}
X_d & Y_d \\
Z_d & W_d
\end{bmatrix}\right)\in\mathbb{B}^d(\mathscr{H}\oplus\mathscr{H}).$$ 
Note that $\mathscr{H}\oplus \mathscr{H}$ is a Hilbert space with the inner product  defined as  $$\langle (
x_1,
x_2
),(
y_1,
y_2
)\rangle=\langle x_1, y_1\rangle + \langle x_2, y_2\rangle,$$ 
for all  $(x_1,x_2)$ and $(y_1,y_2)\in\mathscr{H}\oplus\mathscr{H}.$

Motivated by the existing Euclidean operator radius inequalities (see \cite{BF,SD1,SD,SPK,MSS,P}), here we develop various new inequalities involving joint numerical radius and  joint operator norm of $d$-tuple operators. The inequalities provide lower and upper bounds for the joint numerical radius of $d$-tuple operators, the sum and the product of $d$-tuple operators. Further, we study the joint numerical radius inequalities of $2\times 2$ operator matrices whose entries are $d$-tuple operators, from which we derive some Euclidean operator radius inequalities. We also obtain a joint operator norm inequality of  $2\times 2$ operator matrices.

\section{Joint numerical radius of $d$-tuple operators}

We begin this section with the following proposition, proof of which follows from the definition of the joint operator norm, also see in \cite{KY}. 
 
  \begin{proposition}  
  	If $\mathbf{T}=(T_1,T_2, \ldots,T_d)\in\mathbb{B}^d(\mathscr{H}),$ then $$\|\mathbf{T}\|=\sqrt{\|T_1^*T_1+T_2^*T_2+ \ldots+T_d^*T_d\|}.$$
 \label{leml1}\end{proposition}
Proposition \ref{leml1}  together with the inequality \eqref{T1} leads to the inequality 
\begin{eqnarray}\label{pp00}
	\frac{1}{2\sqrt{d}} \| \mathbf{T}\| \leq w_e(\mathbf{T})\leq \|\mathbf{T}\|.
\end{eqnarray}

We now prove one of our main result, which gives an inequality involving the Euclidean operator radius.

\begin{theorem}
 Let $\mathbf{T}=(T_1,T_2, \ldots,T_d)\in\mathbb{B}^d(\mathscr{H})$.  Then for $x\in\mathscr{H}$, the following inequality holds: $$ \sum_{k=1}^{d}\|T_kx\|^2+\sum_{k=1}^{d}|\langle T_k^2x,x\rangle|\leq2\sqrt{d}\,w_e(\mathbf{T})\left( \sum_{k=1}^{d}\|T_kx\|^2\right)^\frac12\|x\|.$$
\label{lem2}\end{theorem}
\begin{proof}
Let $\lambda_k$ and $\theta_k$ ($ k=1,2,\ldots,d$) be real numbers with $\lambda_k\neq 0$. Then, we have
\begin{eqnarray*}
&& \sum_{k=1}^{d}\|T_kx\|^2+\sum_{k=1}^{d}e^{2i\theta_k}\langle T_k^2x,x\rangle \\ &&=\sum_{k=1}^{d} \Big\{\frac12\langle \lambda_ke^{2i\theta_k}T_k^2x+\lambda_k^{-1}e^{i\theta_k}T_kx,\lambda_ke^{i\theta_k}T_kx+\lambda_k^{-1}x\rangle \\&& \,\,\,\,\,\,-\frac12\langle \lambda_ke^{2i\theta_k}T_k^2x-\lambda_k^{-1}e^{i\theta_k}T_kx,\lambda_ke^{i\theta_k}T_kx-\lambda_k^{-1}x\rangle \Big\}.
\end{eqnarray*}
Hence,
\begin{eqnarray*}
&&\big|\sum_{k=1}^{d}\|T_kx\|^2+\sum_{k=1}^{d}e^{2i\theta_k}\langle T_k^2x,x\rangle\big|\\&&\leq\sum_{k=1}^{d}\frac12\big|\langle \lambda_ke^{2i\theta_k}T_k^2x+\lambda_k^{-1}e^{i\theta_k}T_kx,\lambda_ke^{i\theta_k}T_kx+\lambda_k^{-1}x\rangle\big| \\&&\,\,\,+\sum_{k=1}^{d}\frac12\big|\langle \lambda_ke^{2i\theta_k}T_k^2x-\lambda_k^{-1}e^{i\theta_k}T_kx,\lambda_ke^{i\theta_k}T_kx-\lambda_k^{-1}x\rangle\big| \\&&\leq\sum_{k=1}^{d}\frac12w(T_k)\| \lambda_ke^{i\theta_k}T_kx+\lambda_k^{-1}x\|^2 +\sum_{k=1}^{d}\frac12w(T_k)\| \lambda_ke^{i\theta_k}T_kx-\lambda_k^{-1}x\|^2.
\end{eqnarray*}
Since $|\langle T_kx,x\rangle|\leq\left(\sum_{k=1}^{d}|\langle T_kx,x\rangle|^2\right)^\frac12$ for all $x\in \mathscr{H}$,  $w(T_k)\leq w_e(\mathbf{T})$.
 Thus,
 \begin{eqnarray*}
&&\big|\sum_{k=1}^{d}\|T_kx\|^2+\sum_{k=1}^{d}e^{2i\theta_k}\langle T_k^2x,x\rangle\big|\\
&&\leq\sum_{k=1}^{d}\frac12w_e(\mathbf{T})\| \lambda_ke^{i\theta_k}T_kx+\lambda_k^{-1}x\|^2 +\sum_{k=1}^{d}\frac12w_e(\mathbf{T})\| \lambda_ke^{i\theta_k}T_kx-\lambda_k^{-1}x\|^2\\&&= w_e(\mathbf{T}) \sum_{k=1}^{d}\left\lbrace\frac12\| \lambda_ke^{i\theta_k}T_kx+\lambda_k^{-1}x\|^2 +\frac12\| \lambda_ke^{i\theta_k}T_kx-\lambda_k^{-1}x\|^2\right\rbrace\\
&&= w_e(\mathbf{T})\sum_{k=1}^{d}\left\lbrace\lambda_k^2\|T_kx\|^2+\lambda_k^{-2}\|x\|^2\right\rbrace. 
\end{eqnarray*}
Suppose $T_kx\neq0$ for all $ k=1,2,\dots,d$, and we choose $\theta_k$  in such a way that $e^{2i\theta_k}\langle T_k^2x,x\rangle=|\langle T_k^2x,x\rangle|$ and $\lambda_k=\sqrt{\frac{\|x\|}{\|T_kx\|}}$ for all $ k=1,2, \ldots,d.$ Then, we have
\begin{eqnarray*}
\sum_{k=1}^{d}\|T_kx\|^2+\sum_{k=1}^{d}|\langle T_k^2x,x\rangle|&\leq& 2w_e(\mathbf{T})\sum_{k=1}^{d}\|T_kx\|\|x\|.
\end{eqnarray*}
Therefore, the Cauchy-Schwarz inequality implies that 
\begin{eqnarray*}
	\sum_{k=1}^{d}\|T_kx\|^2+\sum_{k=1}^{d}|\langle T_k^2x,x\rangle|&\leq& 2\sqrt{d} w_e(\mathbf{T})\left(\sum_{k=1}^{d}\|T_kx\|^2\right)^\frac12\|x\|.
\end{eqnarray*}
Also, this inequality holds when $\|T_kx\|=0$ for all or some $k\in \{1,2,\ldots,d\}$. This completes the proof.
\end{proof}

Applying Theorem \ref{lem2} we derive the following corollary. 

\begin{cor}
If $\mathbf{T}=(T_1,T_2, \ldots,T_d)\in\mathbb{B}^d(\mathscr{H}),$ then 
$$ \frac1d w_e(\mathbf{T}^2)\leq  w_e^2(\mathbf{T}).$$

\label{lem3}\end{cor}
\begin{proof}
 From Theorem \ref{lem2} and together with $ \sum_{k=1}^{d}|\langle T_k^2 x,x\rangle|^2\leq\left(\sum_{k=1}^{d}|\langle T_k^2 x,x\rangle|\right)^2$, we have
 \begin{eqnarray}
  \sum_{k=1}^{d}\|T_kx\|^2+\left(\sum_{k=1}^{d}|\langle T_k^2x,x\rangle|^2\right)^\frac12\leq2\sqrt{d}w_e(\mathbf{T})\left( \sum_{k=1}^{d}\|T_kx\|^2\right)^\frac12\|x\|.\label{eqnl3}\end{eqnarray} 
  Taking $\|x\|=1$, we get
  \begin{eqnarray*}
 &&\sum_{k=1}^{d}\|T_kx\|^2+\left(\sum_{k=1}^{d}|\langle T_k^2x,x\rangle|^2\right)^\frac12\leq2\sqrt{d}w_e(\mathbf{T})\left( \sum_{k=1}^{d}\|T_kx\|^2\right)^\frac12.
\end{eqnarray*}
  This implies
  \begin{eqnarray*} \left((\sum_{k=1}^{d}\|T_kx\|^2)^\frac12-\sqrt{d}w_e(\mathbf{T})\right)^2+\left(\sum_{k=1}^{d}|\langle T_k^2x,x\rangle|^2\right)^\frac12\leq d w_e^2(\mathbf{T}).\end{eqnarray*}
  Therefore, 
  $$\left(\sum_{k=1}^{d}|\langle T_k^2x,x\rangle|^2\right)^\frac12\leq d w_e^2(\mathbf{T}).$$
  Taking supremum over all $x\in \mathscr{H}$ with $\|x\|=1$, we get the desired inequality.
\end{proof}

Next, we obtain a refinement of the first inequality in \eqref{pp00}.

\begin{theorem} \label{lemm1}
	Let $\mathbf{T}=(T_1,T_2, \ldots, T_d) \in {\mathbb{B}^d(\mathscr{H})}$ and $\|\mathbf{T}\|\ne0$. Then
		\begin{eqnarray*}
			\frac{1}{2\sqrt{d}}\left\lbrace \|\mathbf{T}\|+\frac{c_e(\mathbf{T^2})}{\|\mathbf{T}\|}\right\rbrace\leq w_e(\mathbf{T}).
		\end{eqnarray*} 
\end{theorem} 

\begin{proof}
Taking $\|x\|=1$ in (\ref{eqnl3}), we get 
\begin{eqnarray*}
  \sum_{k=1}^{d}\|T_kx\|^2+\left(\sum_{k=1}^{d}|\langle T_k^2x,x\rangle|^2\right)^\frac12&\leq&2\sqrt{d}w_e(\mathbf{T})\left( \sum_{k=1}^{d}\|T_kx\|^2\right)^\frac12\\&\leq& 2\sqrt{d}w_e(\mathbf{T})\|\mathbf{T}\|.\end{eqnarray*} 
Hence,	
	\begin{eqnarray*}
  \sum_{k=1}^{d}\|T_kx\|^2&\leq& 2\sqrt{d}w_e(\mathbf{T})\|\mathbf{T}\|-\left(\sum_{k=1}^{d}|\langle T_k^2x,x\rangle|^2\right)^\frac12\\&\leq&2\sqrt{d}w_e(\mathbf{T})\|\mathbf{T}\|-c_e(\mathbf{T^2}).\end{eqnarray*}
  Taking supremum over all $x\in \mathscr{H}$ with $\|x\|=1$,  we get
  $$\|\mathbf{T}\|^2+c_e(\mathbf{T^2})\leq 2\sqrt{d}w_e(\mathbf{T})\|\mathbf{T}\|.$$ 
  This completes the proof.
\end{proof}
  

Now we prove the following inequalities for the joint operator norm of $d$-tuple normal operators. For this purpose we note the well known characterization for normal operator.
An operator $T\in\mathbb{B}(\mathscr{H})$ is normal if and only if $ \|Tx\|=\|T^*x\|$ for all $x\in\mathscr{H}$. 


\begin{theorem} Let $\mathbf{T}=(T_1,T_2, \ldots, T_d) \in {\mathbb{B}^d(\mathscr{H})}$ be a  $d$-tuple normal operator. Then $$\|\mathbf{T^2}\|=\|(T_1^*T_1,T_2^*T_2, \ldots,T_d^*T_d)\|\leq\|\mathbf{T}\|^2=\|\mathbf{T^*}\|^2\leq\sqrt{d}\|\mathbf{T^2}\|.$$
\label{lem5}\end{theorem}
\begin{proof} Let $x\in \mathscr{H}$ with $\|x\|=1$.
Then, we have
\begin{eqnarray*}
\|\mathbf{T^2}\|&=&\|(T_1^2,T_2^2, \ldots,T_d^2)\|
=\underset{\|x\|=1}\sup\left(\sum_{k=1}^{d}\|T_k^2x\|^2\right)^\frac12\\
&=&\underset{\|x\|=1}\sup\left(\sum_{k=1}^{d}\langle T_k^2x, T_k^2x\rangle\right)^\frac12
=\underset{\|x\|=1}\sup\left(\sum_{k=1}^{d}\langle T_kT_kx, T_kT_kx\rangle\right)^\frac12\\
&=&\underset{\|x\|=1}\sup\left(\sum_{k=1}^{d}\langle T_k^*T_kx, T_k^*T_kx\rangle\right)^\frac12\,\,\textit{(since each $T_k$ is normal)}\\
&=&\underset{\|x\|=1}\sup\left(\sum_{k=1}^{d}\| T_k^*T_kx\|^2\right)^\frac12
=\|(T_1^*T_1,T_2^*T_2, \ldots,T_d^*T_d)\|.
\end{eqnarray*}
Now, \begin{eqnarray*}
 \|(T_1^*T_1,T_2^*T_2, \ldots,T_d^*T_d)\|
&=&\underset{\|x\|=1}\sup\left(\sum_{k=1}^{d}\| T_k^*T_kx\|^2\right)^\frac12\\
&\leq&\underset{\|x\|=1}\sup\left(\sum_{k=1}^{d}\| T_k^*\|^2\|T_kx\|^2\right)^\frac12\\ 
&\leq&\underset{\|x\|=1}\sup\left(\sum_{k=1}^{d}\| \mathbf{T}\|^2\|T_kx\|^2\right)^\frac12\\
&&(\textit{since $\|T_kx\|\leq\left(\sum_{k=1}^{d}\|T_kx\|^2\right)^\frac12, $ $\|T_k\|\leq\|\mathbf{T}\|$ for each $k$})\\
&=&\| \mathbf{T}\|\underset{\|x\|=1}\sup\left(\sum_{k=1}^{d}\|T_kx\|^2\right)^\frac12=\| \mathbf{T}\|^2.
\end{eqnarray*}
Also, we have
\begin{eqnarray*}
\| \mathbf{T}\|&=&\underset{\|x\|=1}\sup\left(\sum_{k=1}^{d}\|T_kx\|^2\right)^{1/2}
=\underset{\|x\|=1}\sup\left(\sum_{k=1}^{d}\|T_k^*x\|^2\right)^{1/2}
=\| \mathbf{T^*}\|.
\end{eqnarray*}
Again, \begin{eqnarray*}
\| \mathbf{T}\|^2&=&\underset{\|x\|=1}\sup\left(\sum_{k=1}^{d}\|T_kx\|^2\right)
=\underset{\|x\|=1}\sup\left(\sum_{k=1}^{d}\langle T_kx, T_kx\rangle\right)\\
&=&\underset{\|x\|=1}\sup\left(\sum_{k=1}^{d}\langle T_k^*T_kx, x\rangle\right)
\leq \underset{\|x\|=1}\sup\left(\sum_{k=1}^{d}\| T_k^*T_kx\|\| x\|\right)\\
&\leq&\sqrt{d}\underset{\|x\|=1}\sup\left(\sum_{k=1}^{d}\| T_k^*T_kx\|^2\right)^\frac12\,\,\,(\textit{by Cauchy-Schwarz inequality})\\
&=&\sqrt{d}\underset{\|x\|=1}\sup\left(\sum_{k=1}^{d}\| T_k^2x\|^2\right)^\frac12\,\,(\textit{since each $T_k$ is normal})\\
&=& \sqrt{d}\|\mathbf{T^2}\|.
\end{eqnarray*}
This completes the proof.
\end{proof}

Note that if we take $T_k$ ($k=1,2, \ldots,d$) is a $d\times d$ matrix whose only $(k, k)$ diagonal entries is $1$ and others are zero, then the first inequality in Theorem \ref{lem5} becomes equality.
Also if we take $T_k=\sqrt{d}{I}$ (${I}$ is the $d\times d$ identity matrix) for $k=1,2, \dots,d$, then the second inequality in Theorem \ref{lem5} becomes equality. Thus, the inequalities in Theorem \ref{lem5} are sharp.
  
Now, in the following theorem we develop a power inequality for the joint numerical radius of $d$-tuple operators.

\begin{theorem}
If $\mathbf{T}=(T_1,T_2, \ldots,T_d)\in\mathbb{B}^d(\mathscr{H}),$ then 
$$ w_e(\mathbf{T^n})\leq \sqrt{d} w_e^n(\mathbf{T}).$$
\label{th2}\end{theorem}

\begin{proof} Let $x\in \mathscr{H}$ with $\|x\|=1.$
The inequality  $|\langle T_kx,x\rangle|\leq \left(\sum_{k=1}^{d}|\langle T_kx,x\rangle|^2\right)^\frac12$ implies   $ w(T_k)\leq w_e(\mathbf{T})$ for each $k=1,2,\ldots,d$. Thus, if  $w_e(\mathbf{T})\leq1,$ then $ w(T_k)\leq 1$ for each $ k=1,2, \ldots,d.$
The power inequality \cite{CP} implies that $ w(T_k^n)\leq 1$ for each $ k=1,2, \ldots,d,$  whenever  $ w(T_k)\leq 1.$  Therefore, if $w(T_k)\leq 1 $, then
 \begin{eqnarray*}
w_e(\mathbf{T^n})&=&\underset{\|x\|=1}\sup
\left(\sum_{k=1}^{d}|\langle T_k^nx,x\rangle|^2\right)^\frac12
\leq \left(\sum_{k=1}^{d}\underset{\|x\|=1}\sup|\langle T_k^nx,x\rangle|^2\right)^\frac12
\leq \left(\sum_{k=1}^{d}w^2( T_k^n)\right)^\frac12\\
&\leq&\sqrt{d}.
\end{eqnarray*}
Now, if we take $T_k^{'}=\frac{T_k}{w(\mathbf{T})}$ for all $ k=1,2, \ldots,d,$ then $w_e(\mathbf{T^{'}})= 1$, where $\mathbf{T^{'}}=(T_1^{'},T_2^{'}, \ldots,T_d^{'})$, and so $w(T_k^{'})\leq 1.$ 
Thus, $ w_e(\mathbf{(T^{'})^n})\leq \sqrt{d}$, and this gives $  w_e(\mathbf{T^n})\leq \sqrt{d} w_e^n(\mathbf{T}).$

\end{proof}

Applying Theorem \ref{th2}, we derive the following inequality.

\begin{cor}
Let $\mathbf{T}=(T_1,T_2, \ldots,T_d)\in\mathbb{B}^d(\mathscr{H}).$ If $w_e(\mathbf{T})\leq1,$ then $$\|\mathbf{T^n}\|\leq 2d.$$ 
\end{cor}
\begin{proof}
It follows from the inequality \eqref{pp00} and together with Theorem \ref{th2} that $$ \frac{\|\mathbf{T^n}\|}{2\sqrt{d}}\leq w_e(\mathbf{T^n})\leq \sqrt{d}w_e^n(\mathbf{T})\leq \sqrt{d}.$$
\end{proof}

Now, we obtain the joint numerical radius inequalities for the product of $d$-tuple operators. For this purpose we need the following lemma
in which we prove that the joint operator norm is submultiplicative and the joint numerical radius is subadditive. Though subadditive property of $w_e(\cdot)$ is known, for the convenience of reader we discuss the following proof.

\begin{lemma}
Let $\mathbf{S}=(S_1,S_2, \ldots,S_d),$ $\mathbf{T}=(T_1,T_2, \ldots,T_d) \in \mathbb{B}^d(\mathscr{H}).$ Then the following inequalities hold:  \\$(a) \, \|\mathbf{ST}\|\leq \|\mathbf{S}\|\|\mathbf{T}\|.$ \\
$(b)\, w_e(\mathbf{S}+\mathbf{T})\leq w_e(\mathbf{S})+w_e(\mathbf{T}).$ 
\label{lem6}\end{lemma}
\begin{proof}
Let $x\in \mathscr{H}$ with $\|x\|=1$. Then, we have
\begin{eqnarray*}
\|\mathbf{ST}\|&=&\underset{\|x\|=1}\sup\left(\sum_{k=1}^{d} \|S_kT_kx\|^2\right)^\frac12
\leq \underset{\|x\|=1}\sup\left(\sum_{k=1}^{d} \|S_k\|^2\|T_kx\|^2\right)^\frac12\\
&\leq&\underset{\|x\|=1}\sup\left(\sum_{k=1}^{d} \|\mathbf{S}\|^2\|T_kx\|^2\right)^\frac12\\
&&(\textit{since $\|S_kx\|\leq\left(\sum_{k=1}^{d}\|S_kx\|^2\right)^\frac12 $, $\|S_k\|\leq\|\mathbf{S}\|$ holds for each $k$})\\
&=&\|\mathbf{S}\|\underset{\|x\|=1}\sup\left(\sum_{k=1}^{d} \|T_kx\|^2\right)^\frac12
=\|\mathbf{S}\|\|\mathbf{T}\|.
\end{eqnarray*}

Also, we have
\begin{eqnarray*}
w_e(\mathbf{S}+\mathbf{T})&=& \underset{\|x\|=1}\sup\left(\sum_{k=1}^{d} |\langle (S_k+T_k)x, x\rangle|^2\right)^\frac12\\
&=& \underset{\|x\|=1}\sup\left(\sum_{k=1}^{d} |\langle S_k x, x\rangle+\langle T_k x,x\rangle|^2\right)^\frac12\\
&\leq& \underset{\|x\|=1}\sup\left\lbrace\left(\sum_{k=1}^{d} |\langle S_k x, x\rangle|^2\right)^\frac12+\left(\sum_{k=1}^{d} |\langle T_k x,x\rangle|^2\right)^\frac12\right\rbrace
\\&&\,\,\,\,\,\,(\textit{using Minkowski inequality})\\
&\leq& \underset{\|x\|=1}\sup\left(\sum_{k=1}^{d} |\langle S_k x, x\rangle|^2\right)^\frac12+\underset{\|x\|=1}\sup\left(\sum_{k=1}^{d} |\langle T_k x,x\rangle|^2\right)^\frac12\\
&=& w_e(\mathbf{S})+w_e(\mathbf{T}).
\end{eqnarray*}
\end{proof}


\begin{theorem}
Let $\mathbf{S}=(S_1,S_2, \ldots,S_d)$,  $\mathbf{T}=(T_1,T_2, \ldots,T_d)\in \mathbb{B}^d(\mathscr{H}),$ then $$ w_e(\mathbf{S}\mathbf{T})\leq 4d w_e(\mathbf{S})w_e(\mathbf{T}).$$
\end{theorem}
\begin{proof}
We have, $w_e(\mathbf{S}\mathbf{T})\leq \|\mathbf{S}\mathbf{T}\|\leq\|\mathbf{S}\|\|\mathbf{T}\|\leq 4d w_e(\mathbf{S})w_e(\mathbf{T}),$ where the second inequality is derived from Lemma \ref{lem6} (a) and the third inequality is derived from (\ref{pp00}).
\end{proof}

Further, we develop a joint numerical radius inequality for the product of two $d$-tuple operators $\mathbf{S}$ and $\mathbf{T}$ when $\mathbf{S}\mathbf{T}=\mathbf{T}\mathbf{S}$.

\begin{theorem}
Let $\mathbf{S}=(S_1,S_2, \ldots,S_d)$,  $\mathbf{T}=(T_1,T_2, \ldots,T_d)\in \mathbb{B}^d(\mathscr{H}).$ If $\mathbf{S}\mathbf{T}=\mathbf{T}\mathbf{S}$ (i.e, $S_kT_k=T_kS_k$ for all $k=1,2, \ldots,d$), then $$ w_e(\mathbf{S}\mathbf{T})\leq 2\sqrt{d} w_e(\mathbf{S})w_e(\mathbf{T}).$$
\end{theorem}
\begin{proof}
Suppose $w_e(\mathbf{S})=w_e(\mathbf{T})=1$. Then, we have
 \begin{eqnarray*}
w_e(\mathbf{S}\mathbf{T})&=&w_e\left(\frac14(\mathbf{S}+\mathbf{T})^2-\frac14(\mathbf{S}-\mathbf{T})^2\right)\\
&\leq&\frac14 w_e \left( (\mathbf{S}+\mathbf{T})^2 \right)+\frac14 w_e\left ( (\mathbf{S}-\mathbf{T})^2\right)
\\ && (\textit{using Lemma \ref{lem6} (b) and the fact $w_e(c\mathbf{T})=|c|w_e(\mathbf{T})$})\\
&\leq&\frac{\sqrt{d}}{4} w_e^2(\mathbf{S}+\mathbf{T})+\frac{\sqrt{d}}{4} w_e^2(\mathbf{S}-\mathbf{T})\,\,(\textit{using Theorem \ref{th2}})\\
&\leq&\frac{\sqrt{d}}{4} \left( w_e(\mathbf{S})+w_e(\mathbf{T})\right)^2+\frac{\sqrt{d}}{4} \left( w_e(\mathbf{S})+w_e(\mathbf{T})\right)^2\,\,(\textit{using Lemma \ref{lem6} (b)} )\\
&=& 2\sqrt{d}.
\end{eqnarray*}
This completes the proof.
\end{proof}

Next bound for the product of two $d$-tuple normal operators reads as follows.  

\begin{theorem}
Let $\mathbf{S}=(S_1,S_2, \ldots,S_d)$,  $\mathbf{T}=(T_1,T_2, \ldots,T_d)\in \mathbb{B}^d(\mathscr{H}).$ If $\mathbf{S}$, $\mathbf{T}$ are normal, then  $$ w_e(\mathbf{S}\mathbf{T})\leq  w_e(\mathbf{S})w_e(\mathbf{T}).$$
\end{theorem}
\begin{proof}
	We have
$w_e(\mathbf{S}\mathbf{T})\leq \|\mathbf{S}\mathbf{T}\|\leq\|\mathbf{S}\|\|\mathbf{T}\|=w_e(\mathbf{S})w_e(\mathbf{T}),$
where the last equality follows from   $\|\mathbf{T}\|=w_e(\mathbf{T})$ and $\|\mathbf{S}\|=w_e(\mathbf{S})$, as $\mathbf{T}$, $\mathbf{S}$ both are normal (see \cite{ct}).
\end{proof}

We end this section with the following theorem on joint spectral radius and joint numerical radius.
First we note the following arguments. Following \cite{ct},
the joint approximate point spectrum of a $d$-tuple operator $\mathbf{T}$, denoted by $\sigma_{\pi}(\mathbf{T}),$ is defined  as $$\sigma_{\pi}(\mathbf{T})=\left\lbrace \left(\lambda_1,\lambda_2,\ldots,\lambda_d \right)\in\mathbb{C}^d:\exists \, (x_n)\subseteq\mathscr{H};\|x_n\|=1, \hspace{0.2cm} \lim_{n\rightarrow\infty}\sum_{k=1}^{d}\|(T_k-\lambda_k I) x_n\|=0\right\rbrace .$$
Clearly, this is equivalent to the existence of $(x_n)\subseteq \mathscr{H}$ with  $\|x_n\|=1$ such that $\underset{n\rightarrow\infty}\lim \|(T_k-\lambda_kI)x_n\|=0$ for all $k=1,2, \ldots,d$. For a commuting $d$-tuple operator $\mathbf{T}=(T_1,T_2, \ldots, T_d),$ $\sigma(\mathbf{T})$ denotes the joint spectrum of $\mathbf{T},$ see \cite{D}. It satisfies $\sigma_{\pi}(T)\subseteq \sigma(T)$.
	For a $d$-tuple commuting operator  $\mathbf{T}=(T_1,T_2,\ldots,T_d)\in\mathbb{B}^d(\mathscr{H})$, the non-negative number
	$ r(\mathbf{T})=\sup\left\lbrace \left( \sum_{k=1}^{d} |z_k|^2 \right)^\frac{1}{2}:  (z_1,z_2,\ldots,z_d)\in\sigma(\mathbf{T})\right\rbrace$
	is called joint spectral radius of $ \mathbf{T}$.
For $d$-tuple commuting operator  $\mathbf{T}=(T_1,T_2,\ldots,T_d)\in\mathbb{B}^d(\mathscr{H})$, the inequality $r(\mathbf{T})\leq w_e(\mathbf{T})$ holds.
However, in  \cite{ct}, it is proved that  $r(\mathbf{T})=w_e(\mathbf{T})=\|\mathbf{T}\|$ for $d$-tuple normal operator $\mathbf{T}$.  The

\begin{theorem}\label{jointspectral}
Let $\mathbf{T}=(T_1,T_2, \ldots,T_d)\in\mathbb{B}^d(\mathscr{H})$ be commuting. Then, the following statements are equivalent:\\
(a)  $ r(\mathbf{T})=\|\mathbf{T}\|$.\\
(b)  $w_e(\mathbf{T})=\|\mathbf{T}\|$.
\end{theorem}
\begin{proof}
$(a)\implies(b)$  Let $ r(\mathbf{T})=\|\mathbf{T}\|$.
It easily follows from $ r(\mathbf{T})\leq w_e(\mathbf{T})\leq\|\mathbf{T}\|$ that $w_e(\mathbf{T})=\|\mathbf{T}\|$.
 
$(b)\implies(a)$  Let $w_e(\mathbf{T})=\|\mathbf{T}\|$. Then there exists a sequence $(x_n)\subseteq\mathscr{H}$ with $\|x_n\|=1$ such that $\underset{n\rightarrow\infty}\lim\|(\langle T_1x_n,x_n\rangle,\langle T_2x_n,x_n\rangle, \ldots,\langle T_dx_n,x_n\rangle)\|=\underset{n\rightarrow\infty}\lim \left(\sum_{k=1}^{d}|\langle T_kx_n,x_n\rangle|^2\right)^\frac12 $ $=\|\mathbf{T}\|$.
Without loss of generality assume that  $(\langle T_1x_n,x_n\rangle,\langle T_2x_n,x_n\rangle, \ldots,\langle T_dx_n,x_n\rangle)$  converges to $\lambda=( \lambda_1,\lambda_2, \ldots,\lambda_d)$ and the sequence $\left(\sum_{k=1}^{d}\|T_kx_n\|^2\right)^\frac12 $ converges to $b.$ Then $\|\lambda\|=\|\mathbf{T}\|$. Now,
\begin{eqnarray*}
\sum_{k=1}^{d}\|(T_k-\lambda_k I )x_n\|^2&=&\sum_{k=1}^{d} \|T_k x_n\|^2 +\sum_{k=1}^{d} |\lambda_k|^2-2 Re  \left( \sum_{k=1}^{d}\bar{\lambda_k}\langle T_k x_n, x_n\rangle \right)\\
&\to&  b^2+\|\mathbf{T}\|^2-2\|\lambda\|^2\\
&=& b^2-\|\mathbf{T}\|^2\leq 0. 
\end{eqnarray*}
Hence, $\sum_{k=1}^{d}\|(T_k-\lambda_k{I})x_n\|^2\rightarrow 0$, and so
 $\lambda=( \lambda_1,\lambda_2, \ldots,\lambda_d)\in\sigma_{\pi}(\mathbf{T}).$
This implies $ w_e(\mathbf{T})\leq  r(\mathbf{T})$. 
Hence, $ w_e(\mathbf{T})= r(\mathbf{T})=\|\mathbf{T}\|.$
\end{proof}

\section{Joint numerical radius of $2\times 2$ operator matrices}
\noindent 
We begin this section with proving the following lemma.

  \begin{lemma}
  Let $\mathbf{X}=(X_1,X_2, \ldots,X_d), \mathbf{Y}=(Y_1,Y_2, \ldots,Y_d) \in\mathbb{B}^d(\mathscr{H}).$ Then the following results hold:
  \begin{eqnarray*}
&&(a)\hspace{.2cm}  w_e\left(\begin{bmatrix}
  \mathbf{X} &  0\\
  0 & \mathbf{Y}
\end{bmatrix}\right)=\max\left\lbrace w_e(\mathbf{X}),w_e(\mathbf{Y})\right\rbrace.\\
&&(b)\hspace{.2cm}  \norm{ \begin{bmatrix}
  \mathbf{X} &  0\\
  0 & \mathbf{Y}
\end{bmatrix}} =\max\left\lbrace \|\mathbf{X}\|,\|\mathbf{Y}\|\right\rbrace.\\
&&(c)\hspace{.2cm}  w_e\left(\begin{bmatrix}
     0 & \mathbf{X}\\
  \mathbf{Y} &  0 
\end{bmatrix}\right)=w_e\left(\begin{bmatrix}
 0 & \mathbf{Y} \\
  \mathbf{X} & 0
\end{bmatrix}\right).\\
&&(d)\hspace{.2cm}  w_e\left(\begin{bmatrix}
     0 & \mathbf{X}\\
  \mathbf{Y} &  0 
\end{bmatrix}\right)=w_e\left(\begin{bmatrix}
   0  & \mathbf{X}\\
   \mathbf{e^{i\theta}Y} & 0 
\end{bmatrix}\right)\,\,\,\textit{for all $\theta\in\mathbb{R}$}.\\
&&(e)\hspace{.2cm}  w_e\left(\begin{bmatrix}
  \mathbf{X} &   \mathbf{Y}\\
   \mathbf{Y} & \mathbf{X}
\end{bmatrix}\right)=\max\left\lbrace w_e(\mathbf{X-Y}),w_e(\mathbf{X+Y})\right\rbrace.\\ &&\,\,\,\,\,\,\,\,\,\,\,\,\,\,\, \textit{In particular}, \,\, w_e\left(\begin{bmatrix}
	0 &   \mathbf{Y}\\
	\mathbf{Y} & 0
\end{bmatrix}\right)=w_e(\mathbf{Y}).
\end{eqnarray*} 
 \label{lem1}\end{lemma}
 \begin{proof}
(a) Let $u=(x,y)\in\mathscr{H}\oplus \mathscr{H}$ with $\|u\|=1,$ i.e., $\|x\|^2+\|y\|^2=1.$ Then,
\begin{eqnarray*}
&& \left(\sum_{k=1}^{d}\big|\langle \begin{bmatrix}
 X_k & 0 \\
 0 & Y_k
\end{bmatrix} u,u \rangle\big|^2\right)^\frac12 \\
&=& \left(\sum_{k=1}^{d}|\langle X_k x,x \rangle+\langle Y_k y,y \rangle|^2\right)^\frac12\\
&\leq& \left(\sum_{k=1}^{d}|\langle X_k x,x \rangle|^2\right)^\frac12+\left(\sum_{k=1}^{d}|\langle Y_k y,y \rangle|^2\right)^\frac12\,\,(\textit{using Minkowski inequality})\\
&\leq& w_e(\mathbf{X})\|x\|^2+w_e(\mathbf{Y})\|y\|^2\\
&\leq&\max\left\lbrace w_e(\mathbf{X}),w_e(\mathbf{Y})\right\rbrace\left(\|x\|^2+\|y\|^2\right)
=\max\left\lbrace w_e(\mathbf{X}),w_e(\mathbf{Y})\right\rbrace.
\end{eqnarray*}
Taking supremum over $\|u\|=1$, we get \begin{eqnarray*}
 w_e\left(\begin{bmatrix}
  \mathbf{X} &  0\\
  0 & \mathbf{Y}
\end{bmatrix}\right)\leq \max\left\lbrace w_e(\mathbf{X}),w_e(\mathbf{Y})\right\rbrace.
\end{eqnarray*}
Suppose $u=(x,0)\in\mathscr{H}\oplus \mathscr{H}$ where $\|x\|=1$, then $$\left(\sum_{k=1}^{d}|\langle  \begin{bmatrix}
 X_k & 0 \\
 0 & Y_k
\end{bmatrix} u,u \rangle|^2\right)^\frac12=\left(\sum_{k=1}^{d}|\langle X_k x,x \rangle|^2\right)^\frac12.$$
Taking supremum over $\|x\|=1$, we get $$\underset{\|x\|=1}\sup\left(\sum_{k=1}^{d}|\langle  \begin{bmatrix}
 X_k & 0 \\
 0 & Y_k
\end{bmatrix}u,u \rangle|^2\right)^\frac12= w_e(\mathbf{X}).$$  This implies that $  w_e\left(\begin{bmatrix}
  \mathbf{X} &  0\\
  0 & \mathbf{Y}
\end{bmatrix}\right)\geq  w_e(\mathbf{X})$.
Similarly,  $  w_e\left(\begin{bmatrix}
  \mathbf{X} &  0\\
  0 & \mathbf{Y}
\end{bmatrix}\right)\geq  w_e(\mathbf{Y})$. Hence, $  w_e\left(\begin{bmatrix}
  \mathbf{X} &  0\\
  0 & \mathbf{Y}
\end{bmatrix}\right)\geq \max\left\lbrace w_e(\mathbf{X}),w_e(\mathbf{Y})\right\rbrace.$
This completes the proof of (a).\\
(b) Let $u=(x,y)\in\mathscr{H}\oplus \mathscr{H}$ with $\|u\|=1$, i.e., $\|x\|^2+\|y\|^2=1.$ Then, we have
 \begin{eqnarray*}
\sum_{k=1}^{d}\norm{ \begin{bmatrix}
  X_k &  0\\
  0 & Y_k
\end{bmatrix}u}^2
&=& \sum_{k=1}^{d}\norm{ (X_kx, Y_ky)}^2\\
&=& \sum_{k=1}^{d} \|X_k x\|^2+\|Y_ky\|^2\\
&\leq& \|\mathbf{X}\|^2\|x\|^2+\|\mathbf{Y}\|^2\|y\|^2\\
&\leq&\max\left\lbrace \|\mathbf{X}\|^2,\|\mathbf{Y}\|^2\right\rbrace\left(\|x\|^2+\|y\|^2\right)
=\max\left\lbrace \|\mathbf{X}\|^2,\|\mathbf{Y}\|^2\right\rbrace.
\end{eqnarray*}
Taking supremum over $\|u\|=1$, we get  $$\norm{ \begin{bmatrix}
  \mathbf{X} &  0\\
  0 & \mathbf{Y}
\end{bmatrix}} \leq\max\left\lbrace \|\mathbf{X}\|,\|\mathbf{Y}\|\right\rbrace.$$ 
Now, let $u=(x,0)\in\mathscr{H}\oplus \mathscr{H}$ with $\|x\|=1$, then $$\sum_{k=1}^{d}\norm{ \begin{bmatrix}
  X_k &  0\\
  0 & Y_k
\end{bmatrix}u}^2= \sum_{k=1}^{d}\|X_kx\|^2.$$
Taking supremum over $\|x\|=1$, we get that  $\underset{\|x\|=1}\sup\sum_{k=1}^{d}\norm{ \begin{bmatrix}
  X_k &  0\\
  0 & Y_k
\end{bmatrix}u}^2=\|\mathbf{X}\|^2.$  This implies that $ \norm{ \begin{bmatrix}
  \mathbf{X} &  0\\
  0 & \mathbf{Y}
\end{bmatrix}}\geq\|\mathbf{X}\|.$ 
Similarly, 
$ \norm{ \begin{bmatrix}
  \mathbf{X} &  0\\
  0 & \mathbf{Y}
\end{bmatrix}}\geq\|\mathbf{Y}\|.$ Therefore, $ \norm{ \begin{bmatrix}
  \mathbf{X} &  0\\
  0 & \mathbf{Y}
\end{bmatrix}}\geq\max\left\lbrace\| \mathbf{X}\|,\|\mathbf{Y}\|\right\rbrace.$ This completes the proof of (b).\\
(c) It is easy to verify (see also \cite[Section 2]{P}) that 
\begin{eqnarray}\label{pp0}
	w_e(T_1,T_2,\ldots,T_d)=w_e(U^*T_1U,U^*T_2U,\ldots,U^*T_dU)
\end{eqnarray}  for every unitary operator $U$. The proof (c) follows from \eqref{pp0} by taking $U=\begin{bmatrix}
0 & I \\
I & 0
\end{bmatrix}$.\\
(d) The proof (d) follows from \eqref{pp0} by taking $U=\begin{bmatrix}
I & 0 \\
0 & e^{\frac{i\theta}{2}}I
\end{bmatrix}$. \\
(e) Let  $U=\frac{1}{\sqrt{2}}\begin{bmatrix}
I & I \\
-I & I
\end{bmatrix}$ and $T_k=  \begin{bmatrix}
  X_k &  Y_k\\
  Y_k & X_k
\end{bmatrix}.$ Then $U^*T_kU=\begin{bmatrix}
  X_k- Y_k & 0\\
 0 & Y_k + X_k
\end{bmatrix}.$ Using (a) and \eqref{pp0},  we get $w_e\left(\begin{bmatrix}
  \mathbf{X} &  \mathbf{Y}\\
  \mathbf{Y} & \mathbf{X}
\end{bmatrix}\right)=\max\left\lbrace w_e(\mathbf{X-Y}),w_e(\mathbf{X+Y})\right\rbrace.$ 
In particular, if we take $\mathbf{X}=0$, then $w_e\left(\begin{bmatrix}
	0 &  \mathbf{Y}\\
	\mathbf{Y} & 0
\end{bmatrix}\right)= w_e(\mathbf{Y}).$ This completes the proof (e).
\end{proof}

Next we develop an upper bound for the joint numerical radius of $2\times 2$ operator matrices whose entries are $d$-tuple operators.

\begin{theorem}
Let $\mathbf{X}=(X_1,X_2, \ldots,X_d)$, $\mathbf{Y}=(Y_1,Y_2, \ldots,Y_d)$,  $\mathbf{Z}=(Z_1,Z_2, \ldots,Z_d),$  $\mathbf{W}=(W_1,W_2, \ldots,W_d)\in\mathbb{B}^d(\mathscr{H})$. Then $$ w_e\left(\begin{bmatrix}
\mathbf{X} & \mathbf{Y} \\
\mathbf{Z} & \mathbf{W} 
\end{bmatrix}\right)\leq w \left(\begin{bmatrix}
w_e(\mathbf{X}) & \|\mathbf{Y}\| \\
\|\mathbf{Z}\| & w_e(\mathbf{W}) 
\end{bmatrix}\right).$$  
\label{them2}\end{theorem}
\begin{proof}
Let $u=(x,y)\in\mathscr{H}\oplus \mathscr{H}$ with $\|u\|=1$, i.e., $\|x\|^2+\|y\|^2=1.$ Now,
\begin{eqnarray*}
&&\left(\sum_{k=1}^{d} \big| \langle \begin{bmatrix}
 X_k & Y_k \\
 Z_k & W_k
\end{bmatrix}u , u \rangle\big|^2\right)^\frac12\\
&=& \left(\sum_{k=1}^{d} \big| \langle (X_k x+Y_k y, Z_k x + W_k y), (x,y) \rangle\big|^2\right)^\frac12\\
&\leq& \left(\sum_{k=1}^{d} \big| \langle X_k x,x\rangle+\langle W_k y, y \rangle\big|^2\right)^\frac12 + \left(\sum_{k=1}^{d} \big| \langle Y_k y,x\rangle+\langle Z_k x, y \rangle\big|^2\right)^\frac12\\
&& \,\,\,\,\,\,\,\,\,(\textit{using Minkowski inequality})\\
&\leq& \left(\sum_{k=1}^{d} \big| \langle X_k x,x\rangle\big|^2\right)^\frac12+\left(\sum_{k=1}^{d} \big| \langle W_k y,y\rangle\big|^2\right)^\frac12+\left(\sum_{k=1}^{d} \big| \langle Z_k x,y\rangle\big|^2\right)^\frac12+ \left(\sum_{k=1}^{d} \big| \langle Y_k y,x\rangle\big|^2\right)^\frac12\\
&&\,\,\,\,\,\,\,\,\,(\textit{using Minkowski inequality})\\
&\leq& w_e(\mathbf{X})\|x\|^2+ w_e(\mathbf{W})\|y\|^2+\left(\sum_{k=1}^{d} \|Z_k x\|^2\|y\|^2\right)^\frac12+\left(\sum_{k=1}^{d} \|Y_k y\|^2\|x\|^2\right)^\frac12\\
&\leq& w_e(\mathbf{X})\|x\|^2+ w_e(\mathbf{W})\|y\|^2+\|\mathbf{Z}\|\|x\|\|y\|+ \|\mathbf{Y}\|\|y\|\|x\|\\
&=& \langle \begin{bmatrix}
 w_e(\mathbf{X}) & \|\mathbf{Y}\| \\
 \|\mathbf{Z}\| & w_e(\mathbf{W})
\end{bmatrix}\widetilde{x},\widetilde{x}\rangle,\,\,\textit{where $\widetilde{x}=
(\|x\|,
\|y\|
)\in \mathbb{C}^2$}.
\end{eqnarray*} Thus, \begin{eqnarray*}
  w_e\left(\begin{bmatrix}
\mathbf{X} & \mathbf{Y} \\
\mathbf{Z} & \mathbf{W}
\end{bmatrix}\right)&=&\sup\left\lbrace\left(\sum_{k=1}^{d} \big| \langle \begin{bmatrix}
 X_k & Y_k \\
 Z_k & W_k
\end{bmatrix}u , u \rangle\big|^2\right)^\frac12:u\in\mathscr{H}\oplus \mathscr{H},\|u\|=1\right\rbrace\\
&\leq&w \left(\begin{bmatrix}
w_e(\mathbf{X}) & \|\mathbf{Y}\| \\
\|\mathbf{Z}\| & w_e(\mathbf{W}) 
\end{bmatrix}\right).
\end{eqnarray*}

\end{proof}

Note that $w_e(\mathbf{X})\leq \|\mathbf{X}\|, w_e(\mathbf{W})\leq \|\mathbf{W}\|$ and $w([a_{ij}])\leq w([b_{ij}]),$ for all $0\leq a_{ij}\leq b_{ij}$. Therefore, the following corollary is immediate from Theorem \ref{them2}.

\begin{cor}
	Let  $\mathbf{X}=(X_1,X_2, \ldots,X_d)$, $\mathbf{Y}=(Y_1,Y_2, \ldots,Y_d)$,  $\mathbf{Z}=(Z_1,Z_2, \ldots,Z_d),$  $\mathbf{W}=(W_1,W_2, \ldots,W_d) \in\mathbb{B}^d(\mathscr{H})$. Then $$ w_e\left(\begin{bmatrix}
		\mathbf{X} & \mathbf{Y} \\
		\mathbf{Z} & \mathbf{W}
	\end{bmatrix}\right)\leq w \left(\begin{bmatrix}
		\|\mathbf{X}\| & \|\mathbf{Y}\| \\
		\|\mathbf{Z}\| & \|\mathbf{W}\| 
	\end{bmatrix}\right).$$  
	\label{them1}\end{cor}

It should be mentioned here that Theorem \ref{them2} gives better bound than that in Corollary \ref{them1}.
To prove the next result we need the following lemma. 
\begin{lemma}\cite[p. 44]{H}
Let $B=[b_{ij}]$ be an $n\times n $ matrix such that $b_{ij}\geq0$ for all $i,j=1,2,...,n.$ Then $$
w(B)=r\left(\begin{bmatrix}
\frac{b_{ij}+b_{ji}}{2}
\end{bmatrix}\right),$$ where $r(\cdot)$ denotes the spectral radius. 
\label{lem11}\end{lemma}

Applying Theorem \ref{them2} and Corollary \ref{them1}, and  using Lemma \ref{lem11}, we obtain the following two corollaries.

\begin{cor}
	Let  $\mathbf{X}=(X_1,X_2, \ldots,X_d)$, $\mathbf{Y}=(Y_1,Y_2,\ldots,Y_d)$,  $\mathbf{Z}=(Z_1,Z_2, \ldots,Z_d),$  $\mathbf{W}=(W_1,W_2, \ldots,W_d)\in\mathbb{B}^d(\mathscr{H})$. Then $$ w_e\left(\begin{bmatrix}
		\mathbf{X} & \mathbf{Y} \\
		\mathbf{Z} & \mathbf{W}
	\end{bmatrix}\right)\leq r\left([c_{ij}]\right)=\frac12\left(w_e(\mathbf{X})+w_e(\mathbf{W})+\sqrt{(w_e(\mathbf{X})-w_e(\mathbf{W}))^2+(\|\mathbf{Y}\|+\|\mathbf{Z}\|)^2}\right),$$ where $c_{11}=w_e(\mathbf{X})$, $c_{12}=c_{21}=\frac{\|\mathbf{Y}\|+\|\mathbf{Z}\|}{2}$, $c_{22}=w_e(\mathbf{W})$.  
\end{cor}

\begin{cor}
Let $\mathbf{X}=(X_1,X_2, \ldots,X_d)$, $\mathbf{Y}=(Y_1,Y_2, \ldots,Y_d)$,  $\mathbf{Z}=(Z_1,Z_2, \ldots,Z_d),$  $\mathbf{W}=(W_1,W_2, \ldots,W_d)\in\mathbb{B}^d(\mathscr{H})$. Then $$ w_e\left(\begin{bmatrix}
\mathbf{X} & \mathbf{Y} \\
\mathbf{Z} & \mathbf{W}
\end{bmatrix}\right)\leq r\left([c_{ij}]\right)=\frac12\left(\|\mathbf{X}\|+\|\mathbf{W}\|+\sqrt{(\|\mathbf{X}\|-\|\mathbf{W}\|)^2+(\|\mathbf{Y}\|+\|\mathbf{Z}\|)^2}\right),$$ where $c_{11}=\|\mathbf{X}\|$, $c_{12}=c_{21}=\frac{\|\mathbf{Y}\|+\|\mathbf{Z}\|}{2}$, $c_{22}=\|\mathbf{W}\|$.  
\end{cor}

Next, by using the power inequality obtained in Theorem \ref{th2}, we develop a lower bound for the joint numerical radius of   $2 \times2 $ operator matrices whose entries are $d$-tuple operators.

\begin{theorem}
Let $\mathbf{X}=(X_1,X_2, \ldots,X_d)$, $\mathbf{Y}=(Y_1,Y_2, \ldots,Y_d)\in\mathbb{B}^d(\mathscr{H})$. Then $$\sqrt[2n]{ \frac{1}{\sqrt{d}}\max \{w_e((\mathbf{XY})^n),w_e((\mathbf{Y}\mathbf{X})^n)}\}\leq w_e\left(\begin{bmatrix}
0 & \mathbf{X} \\
\mathbf{Y} & 0
\end{bmatrix}\right).$$
\end{theorem}
\begin{proof}
Let $\mathbf{T}=\begin{bmatrix}
0 & \mathbf{X} \\
\mathbf{Y} & 0
\end{bmatrix}.$  Then  $\mathbf{T}^{2n}=\begin{bmatrix}
 (\mathbf{XY})^n & 0 \\
 0 &  (\mathbf{YX})^n
\end{bmatrix}$ for all $n=1,2,3, \ldots$.
Using Lemma \ref{lem1} (a) and  Theorem \ref{th2}, we get $$\max \{w_e((\mathbf{X}\mathbf{Y})^n),w_e((\mathbf{Y}\mathbf{X})^n)\}=w_e( \mathbf{T}^{2n})\leq\sqrt{d} w_e^{2n}( \mathbf{T}).$$   This completes the proof.  
\end{proof}

Next we prove the following lower and upper bounds.

\begin{theorem}
Let $\mathbf{X}=(X_1,X_2, \ldots,X_d)$, $\mathbf{Y}=(Y_1,Y_2, \ldots,Y_d)\in\mathbb{B}^d(\mathscr{H}).$ Then 
\begin{eqnarray*}
 \frac{1}{2} \max\left\lbrace w_e(\mathbf{X}+\mathbf{Y}),w_e(\mathbf{X}-\mathbf{Y})\right\rbrace \leq w_e\left(\begin{bmatrix}
	0 & \mathbf{X} \\
	\mathbf{Y} & 0
\end{bmatrix}\right) \leq \frac{1}{2}(w_e(\mathbf{X}+\mathbf{Y})+w_e(\mathbf{X}-\mathbf{Y})).	
\end{eqnarray*}
\label{th5}\end{theorem} 
\begin{proof}
It follows from Lemma \ref{lem1} (e)  that 
\begin{eqnarray}
\nonumber w_e(\mathbf{X}+\mathbf{Y})&&=w_e\left(\begin{bmatrix}
0 & \mathbf{X+Y} \\
\nonumber \mathbf{X+Y} & 0
\end{bmatrix}\right)\\
&=& w_e\left(\begin{bmatrix}
0 & \mathbf{X} \\
\nonumber \mathbf{Y} & 0
\end{bmatrix}+\begin{bmatrix}
0 & \mathbf{Y} \\
\nonumber \mathbf{X} & 0
\end{bmatrix}\right)\\
\nonumber &\leq&  w_e\left(\begin{bmatrix}
0 & \mathbf{X} \\
\mathbf{Y} & 0
\end{bmatrix}\right)+w_e\left(\begin{bmatrix}
0 & \mathbf{Y} \\
\mathbf{X} & 0
\end{bmatrix}\right)\,\,(\textit{using Lemma \ref{lem6})}\\
&=& 2 w_e\left(\begin{bmatrix}
0 & \mathbf{X} \\
\mathbf{Y} & 0
\end{bmatrix}\right)\,\,\,(\textit{using Lemma \ref{lem1} (c))}.
\label{eqn2}\end{eqnarray}
Replacing $\mathbf{Y}$ by $ - \mathbf{Y}$, we have 
\begin{eqnarray}
w_e(\mathbf{X}-\mathbf{Y})\leq 2 w_e\left(\begin{bmatrix}
0 & \mathbf{X} \\
-\mathbf{Y} & 0
\end{bmatrix}\right)= 2 w_e\left(\begin{bmatrix}
0 & \mathbf{X} \\
\mathbf{Y} & 0
\end{bmatrix}\right).\label{eqn3}\end{eqnarray}
Therefore, the first inequality follows from (\ref{eqn2}) and (\ref{eqn3}).
To prove the second inequality, consider an unitary operator $U=\frac{1}{\sqrt{2}}\begin{bmatrix}
I & -I \\
I &  I
\end{bmatrix}$.
Then we have,
\begin{eqnarray*}
 w_e\left(\begin{bmatrix}
0 & \mathbf{X} \\
\mathbf{Y} & 0
\end{bmatrix}\right)&=& w_e\left(U^*\begin{bmatrix}
0 & \mathbf{X} \\
\mathbf{Y} & 0
\end{bmatrix}U\right)\\
&=& \frac12 w_e\left(\begin{bmatrix}
\mathbf{X+Y} & \mathbf{X-Y} \\
-(\mathbf{X-Y}) & -(\mathbf{X+Y})
\end{bmatrix}\right)\\
&=&\frac12 w_e\left(\begin{bmatrix}
\mathbf{X+Y} & 0 \\
0 & -(\mathbf{X+Y})
\end{bmatrix}+\begin{bmatrix}
0 & \mathbf{X-Y} \\
-(\mathbf{X-Y}) &0
\end{bmatrix}\right)\\
&\leq&\frac12 w_e\left(\begin{bmatrix}
\mathbf{X+Y} & 0 \\
0 & -(\mathbf{X+Y})
\end{bmatrix}\right)+\frac12 w_e\left(\begin{bmatrix}
0 & \mathbf{X-Y} \\
-(\mathbf{X-Y}) & 0
\end{bmatrix}\right)\\
&&\,\,(\textit{using Lemma \ref{lem6})}\\
&=&\frac{w_e(\mathbf{X}+\mathbf{Y})+w_e(\mathbf{X}-\mathbf{Y})}{2}\,\,\,\,\,\,(\textit{using Lemma \ref{lem1})}.
\end{eqnarray*}  
This completes the proof.
\label{them6}\end{proof}

As application of  Theorem \ref{th5}, we derive the following inequalities.
\begin{cor}
Let $ \mathbf{T}=(T_1,T_2, \ldots,T_d)\in\mathbb{B}^d(\mathscr{H})$ and let $T_k=X_k+iY_k$ be the Cartesian decomposition, for all $k=1,2, \ldots,d$. Then $$ \frac{w_e(\mathbf{T})}{2}\leq w_e\left(\begin{bmatrix}
0 & \mathbf{X} \\
e^{i\theta}\mathbf{Y} & 0
\end{bmatrix}\right)\leq w_e(\mathbf{T}),$$ for all $\theta\in\mathbb{R}$. 
\end{cor}
\begin{proof}
Replacing $\mathbf{Y}$ by $ i\mathbf{Y}$ in Theorem \ref{th5}, and then using Lemma \ref{lem1}, we have $$\frac{\max\left\lbrace w_e(\mathbf{X}+i\mathbf{Y}),w_e(\mathbf{X}-i\mathbf{Y})\right\rbrace}{2}\leq w_e\left(\begin{bmatrix}
0 & \mathbf{X} \\
e^{i\theta}\mathbf{Y} & 0
\end{bmatrix}\right)\leq \frac{w_e(\mathbf{X}+i\mathbf{Y})+w_e(\mathbf{X}-i\mathbf{Y})}{2}.$$
This implies  \begin{eqnarray}\frac{\max\left\lbrace w_e(\mathbf{T}),w_e(\mathbf{T^*})\right\rbrace}{2}\leq w_e\left(\begin{bmatrix}
0 & \mathbf{X} \\
e^{i\theta}\mathbf{Y} & 0
\end{bmatrix}\right)\leq \frac{w_e(\mathbf{T})+w_e(\mathbf{T^*})}{2}.\label{eqn4}\end{eqnarray} Since $w_e(\mathbf{T})=w_e(\mathbf{T^*})$, the proof follows from (\ref{eqn4}).
\end{proof}

In view of the expression for $w_e\left(\begin{bmatrix}
\mathbf{X} & \mathbf{Y} \\
\mathbf{Y} & \mathbf{X}
\end{bmatrix}\right)$ obtained in Lemma \ref{lem1}, it is natural to look for  similar expression for  $w_e\left(\begin{bmatrix}
\mathbf{X} & \mathbf{Y} \\
\mathbf{-Y} & \mathbf{-X}
\end{bmatrix}\right)$. To obtain this first we need to prove the following lemma.
\begin{lemma}
Let $\mathbf{X}=(X_1,X_2, \ldots,X_d)$, $\mathbf{Y}=(Y_1,Y_2, \ldots ,Y_d)$,  $\mathbf{Z}=(Z_1,Z_2, \ldots,Z_d),$  $\mathbf{W}=(W_1,W_2, \ldots,W_d)\in\mathbb{B}^d(\mathscr{H})$. Then $$ w_e\left(\begin{bmatrix}
\mathbf{X} & \mathbf{Y} \\
\mathbf{Z} & \mathbf{W} 
\end{bmatrix}\right)\geq w_e\left(\begin{bmatrix}
\mathbf{X} & 0 \\
0 & \mathbf{W} 
\end{bmatrix}\right)$$ and $$ w_e\left(\begin{bmatrix}
\mathbf{X} & \mathbf{Y} \\
\mathbf{Z} & \mathbf{W} 
\end{bmatrix}\right)\geq w_e\left(\begin{bmatrix}
0 &\mathbf{Y} \\
\mathbf{Z} & 0 
\end{bmatrix}\right).$$ 

\label{lem7}\end{lemma}

\begin{proof}
Let $u=(x,0)\in\mathscr{H}\oplus \mathscr{H}$ with $\|u\|=1$, i.e., $\|x\|=1.$ Now, we have 
\begin{eqnarray*}
\left(\sum_{k=1}^{d}\big|\langle\begin{bmatrix}
\mathbf{X} & \mathbf{Y} \\
\mathbf{Z} & \mathbf{W} 
\end{bmatrix}u,u\rangle\big|^2\right)^\frac12
&=&\left(\sum_{k=1}^{d}|\langle (X_kx,Z_kx),(x,0)\rangle|^2 \right)^\frac12\\
&=& \left(\sum_{k=1}^{d}|\langle X_kx,x\rangle|^2 \right)^\frac12.
\end{eqnarray*}
Taking supremum over $\|u\|=1$, we get $$\underset{\|u\|=1}\sup\left(\sum_{k=1}^{d}\big|\langle\begin{bmatrix}
\mathbf{X} & \mathbf{Y} \\
\mathbf{Z} & \mathbf{W} 
\end{bmatrix}u,u\rangle\big|^2\right)^\frac12=\underset{\|x\|=1}\sup\left(\sum_{k=1}^{d}|\langle X_kx,x\rangle|^2 \right)^\frac12=w_e(\mathbf{X}).$$
This gives,      \begin{eqnarray}
  w_e(\mathbf{X})\leq w_e\left(\begin{bmatrix}
\mathbf{X} & \mathbf{Y} \\
\mathbf{Z} & \mathbf{W} 
\end{bmatrix}\right).\label{eqn5}\end{eqnarray}\\
Similarly, \begin{eqnarray} w_e(\mathbf{Y})\leq w_e\left(\begin{bmatrix}
\mathbf{X} & \mathbf{Y} \\
\mathbf{Z} & \mathbf{W} 
\end{bmatrix}\right).\label{eqn6}\end{eqnarray}
Therefore, the desired first inequality follows from (\ref{eqn5}) and (\ref{eqn6}) together with Lemma \ref{lem1} (a).
To prove the second inequality, we write
$\begin{bmatrix}
0 & \mathbf{Y} \\
\mathbf{Z} &0 
\end{bmatrix}=\frac12 \begin{bmatrix}
\mathbf{X} & \mathbf{Y} \\
\mathbf{Z} & \mathbf{W} 
\end{bmatrix} +\frac12 \begin{bmatrix}
-\mathbf{X} & \mathbf{Y} \\
\mathbf{Z} & -\mathbf{W} 
\end{bmatrix}.$ It follows from Lemma \ref{lem6}
that $w_e\left(\begin{bmatrix}
0 & \mathbf{Y} \\
\mathbf{Z} &0 
\end{bmatrix}\right)\leq\frac12 w_e\left( \begin{bmatrix}
\mathbf{X} & \mathbf{Y} \\
\mathbf{Z} & \mathbf{W} 
\end{bmatrix}\right) +\frac12 w_e\left( \begin{bmatrix}
-\mathbf{X} & \mathbf{Y} \\
\mathbf{Z} & -\mathbf{W} 
\end{bmatrix}\right).$
By considering the unitary operator $U=\begin{bmatrix}
0 & -I \\
I & 0 
\end{bmatrix}$, we have $ U^*\begin{bmatrix}
-\mathbf{X} & \mathbf{Y} \\
\mathbf{Z} & -\mathbf{W} 
\end{bmatrix}U=\begin{bmatrix}
-\mathbf{W} & -\mathbf{Z} \\
-\mathbf{Y} & -\mathbf{X} 
\end{bmatrix}$, and using the property $w_e(U^*\mathbf{X}U)=w_e(\mathbf{X})$ we infer that
\begin{eqnarray}
 w_e\left(\begin{bmatrix}
0 & \mathbf{Y} \\
\mathbf{Z} &0 
\end{bmatrix}\right)\leq\frac12 w_e\left( \begin{bmatrix}
\mathbf{X} & \mathbf{Y} \\
\mathbf{Z} & \mathbf{W} 
\end{bmatrix}\right) +\frac12 w_e\left( \begin{bmatrix}
-\mathbf{W} & -\mathbf{Z} \\
-\mathbf{Y} & -\mathbf{X} 
\end{bmatrix}\right).\label{1}\end{eqnarray}\\ 
Again, considering the unitary operator $ U=\begin{bmatrix}
0 & I \\
I & 0
\end{bmatrix},$ we have $U^*\begin{bmatrix}
-\mathbf{W} & -\mathbf{Z} \\
-\mathbf{Y} & -\mathbf{X} 
\end{bmatrix}U=\begin{bmatrix}
-\mathbf{X} & -\mathbf{Y} \\
-\mathbf{Z} & -\mathbf{W} 
\end{bmatrix}$ and $
 w_e\left( \begin{bmatrix}
-\mathbf{W} & -\mathbf{Z} \\
-\mathbf{Y} & -\mathbf{X} 
\end{bmatrix}\right)=w_e\left(\begin{bmatrix}
-\mathbf{X} & -\mathbf{Y} \\
-\mathbf{Z} & -\mathbf{W} 
\end{bmatrix}\right)=w_e\left(\begin{bmatrix}
\mathbf{X} & \mathbf{Y} \\
\mathbf{Z} & \mathbf{W} 
\end{bmatrix}\right).$ By
 using this argument, the required second inequality follows from (\ref{1}).

\end{proof}

\begin{theorem}
Let  $\mathbf{X}=(X_1,X_2, \ldots,X_d)$, $\mathbf{Y}=(Y_1,Y_2, \ldots,Y_d)\in\mathbb{B}^d(\mathscr{H}).$  Then $$\max\left\lbrace w_e(\mathbf{X}),w_e(\mathbf{Y})\right\rbrace\leq w_e\left(\begin{bmatrix}
\mathbf{X} & \mathbf{Y} \\
-\mathbf{Y} & -\mathbf{X}
\end{bmatrix}\right)\leq w_e(\mathbf{X})+ w_e(\mathbf{Y}).$$
\label{th7}\end{theorem}
\begin{proof}
The first inequality follows from Lemma \ref{lem7} together with Lemma \ref{lem1}.  For the other part, 
\begin{eqnarray*}
w_e\left(\begin{bmatrix}
\mathbf{X} & \mathbf{Y} \\
-\mathbf{Y} & -\mathbf{X}
\end{bmatrix}\right)&\leq& w_e\left(\begin{bmatrix}
\mathbf{X} & 0 \\
0 & -\mathbf{X}
\end{bmatrix}+\begin{bmatrix}
0 & \mathbf{Y} \\
-\mathbf{Y} & o
\end{bmatrix}\right)\\
& \leq &w_e\left(\begin{bmatrix}
\mathbf{X} & 0 \\
0 & -\mathbf{X}
\end{bmatrix}\right)+ w_e\left(\begin{bmatrix}
0 & \mathbf{Y} \\
-\mathbf{Y} & o
\end{bmatrix}\right) \,\,\,(\textit{using Lemma $\ref{lem6}$})\\
&=&  w_e(\mathbf{X})+ w_e(\mathbf{Y}).
\end{eqnarray*} 

\end{proof}

In particular, taking  $\mathbf{Y}=\mathbf{X}$ in Theorem \ref{th7}, we derive the following inequality.
\begin{cor} If $\mathbf{X} \in\mathbb{B}^d(\mathscr{H})$, then
\begin{eqnarray*}
	w_e(\mathbf{X})\leq w_e\left(\begin{bmatrix}
		\mathbf{X} & \mathbf{X} \\
		-\mathbf{X} & -\mathbf{X}
	\end{bmatrix}\right)\leq2 w_e(\mathbf{X}).
\end{eqnarray*}
\end{cor}

Next we obtain the following lower and upper bounds for the joint numerical radius of $2\times 2$ operator matrices whose entries are $d$-tuple operators.

\begin{theorem}
Let $\mathbf{X}=(X_1,X_2, \ldots,X_d)$, $\mathbf{Y}=(Y_1,Y_2, \ldots,Y_d)$,  $\mathbf{Z}=(Z_1,Z_2, \ldots,Z_d),$  $\mathbf{W}=(W_1,W_2, \ldots,W_d)\in\mathbb{B}^d(\mathscr{H})$.  Then $$ w_e\left(\begin{bmatrix}
\mathbf{X} & \mathbf{Y} \\
\mathbf{Z} & \mathbf{W}
\end{bmatrix}\right)\geq\max\left\lbrace w_e(\mathbf{X}),w_e(\mathbf{W}),w_e\left(\frac{\mathbf{Y}+\mathbf{Z}}{2}\right),w_e\left(\frac{\mathbf{Y}-\mathbf{Z}}{2}\right)\right\rbrace$$ and  $$ w_e\left(\begin{bmatrix}
\mathbf{X} & \mathbf{Y} \\
\mathbf{Z} & \mathbf{W}
\end{bmatrix}\right)\leq\max\left\lbrace w_e(\mathbf{X}),w_e(\mathbf{W})\right\rbrace+w_e\left(\frac{\mathbf{Y}+\mathbf{Z}}{2}\right)+w_e\left(\frac{\mathbf{Y}-\mathbf{Z}}{2}\right).$$
\end{theorem}
\begin{proof}
It follows from Lemma \ref{lem7} that \begin{eqnarray*}
&&  w_e\left(\begin{bmatrix}
\mathbf{X} & \mathbf{Y} \\
\mathbf{Z} & \mathbf{W}
\end{bmatrix}\right)\\
&\geq&\max\left\lbrace w_e\left(\begin{bmatrix}
\mathbf{X} & 0 \\
0 & \mathbf{W} 
\end{bmatrix}\right),\left(\begin{bmatrix}
0 &\mathbf{Y} \\
\mathbf{Z} & 0 
\end{bmatrix}\right)\right\rbrace\\
&=& \max\left\lbrace w_e(\mathbf{X}),w_e(\mathbf{X}),w_e\left(\begin{bmatrix}
0 &\mathbf{Y} \\
\mathbf{Z} & 0 
\end{bmatrix}\right)\right\rbrace \,\,\,(\textit{using Lemma \ref{lem1}(a)})\\
&\geq&\max\left\lbrace w_e(\mathbf{X}),w_e(\mathbf{W}),w_e\left(\frac{\mathbf{Y}+\mathbf{Z}}{2}\right),w_e\left(\frac{\mathbf{Y}-\mathbf{Z}}{2}\right)\right\rbrace \,\,\,(\textit{using Theorem \ref{th5})
}. 
\end{eqnarray*}
Again, it follows from Lemma \ref{lem1}(a) and Theorem \ref{th5} that \begin{eqnarray*}
w_e\left(\begin{bmatrix}
\mathbf{X} & \mathbf{Y} \\
\mathbf{Z} & \mathbf{W}
\end{bmatrix}\right)&\leq& w_e\left(\begin{bmatrix}
\mathbf{X} & 0\\
0 & \mathbf{W}
\end{bmatrix}+\begin{bmatrix}
0 & \mathbf{Y} \\
\mathbf{Z} &0
\end{bmatrix}\right)\\
&\leq&w_e\left(\begin{bmatrix}
\mathbf{X} & 0\\
0 & \mathbf{W}
\end{bmatrix}\right)+w_e\left(\begin{bmatrix}
0 & \mathbf{Y} \\
\mathbf{Z} &0
\end{bmatrix}\right)\,\,(\textit{using Lemma \ref{lem6})} \\
&=&\max\left\lbrace w_e(\mathbf{X}),w_e(\mathbf{W})\right\rbrace+w_e\left(\frac{\mathbf{Y}+\mathbf{Z}}{2}\right)+w_e\left(\frac{\mathbf{Y}-\mathbf{Z}}{2}\right).
\end{eqnarray*}
This completes the proof.
\end{proof}

Now by applying the results obtained above, and  using the identity
 \begin{eqnarray}
\frac{a+b}{2}=\max\{a,b\} -\frac{|a-b|}{2} \,\, \,\text{for $a,b\geq 0$},
\label{eqn7}\end{eqnarray}
we develop the following inequality.
 
\begin{theorem}
	Let $\mathbf{X}=(X_1,X_2, \ldots,X_d)$, $\mathbf{Y}=(Y_1,Y_2, \ldots,Y_d)\in\mathbb{B}^d(\mathscr{H})$. Then $$ w_e\left(\begin{bmatrix}
0 & \mathbf{X} \\
\mathbf{Y} & 0
\end{bmatrix}\right)+\frac{|w_e(\mathbf{X+Y})-w_e(\mathbf{X-Y})|}{2}\leq w_e(\mathbf{X})+w_e(\mathbf{Y}).$$\end{theorem}
\begin{proof}
From Theorem \ref{th5} and identity  \eqref{eqn7}, we have \begin{eqnarray*}
 w_e\left(\begin{bmatrix}
0 & \mathbf{X} \\
\mathbf{Y} & 0
\end{bmatrix}\right)&\leq&\frac{w_e(\mathbf{X}+\mathbf{Y})+w_e(\mathbf{X}-\mathbf{Y})}{2}\\
&=&\max\left\lbrace w_e(\mathbf{X}+\mathbf{Y}),w_e(\mathbf{X}-\mathbf{Y})\right\rbrace- \frac{|w_e(\mathbf{X+Y})-w_e(\mathbf{X-Y})|}{2}\\
&\leq&  w_e(\mathbf{X})+w_e(\mathbf{Y})-\frac{|w_e(\mathbf{X+Y})-w_e(\mathbf{X-Y})|}{2}.
\end{eqnarray*}
This gives the desired inequality.
\end{proof}

Next theorem reads as follows:

\begin{theorem}
Let  $\mathbf{X}=(X_1,X_2, \ldots,X_d)$, $\mathbf{Y}=(Y_1,Y_2, \ldots,Y_d),$  $\mathbf{A}=(A_1,A_2, \ldots,A_d)$,  $\mathbf{B}=(B_1,B_2, \ldots,B_d)\in\mathbb{B}^d(\mathscr{H}).$ Then $$ w_e(\mathbf{A^*XB+B^*YA})\leq 2 \|\mathbf{A}\||\mathbf{B}\|w_e\left(\begin{bmatrix}
0 & \mathbf{X} \\
\mathbf{Y} & 0
\end{bmatrix}\right).$$ In particular,  $$w_e(\mathbf{A^*XB+B^*XA})\leq 2 \|\mathbf{A}\||\mathbf{B}\|w_e(\mathbf{X}).$$
\end{theorem}
\begin{proof}
Let $x,y\in\mathscr{H} $ be non zero, and  let $ z=\frac{1}{\sqrt{\|x\|^2+\|y\|^2}}(
x,
y)
$. Then $z$ is an unit vector in $\mathscr{H}\oplus\mathscr{H}$, and so we have, \begin{eqnarray*}
w_e\left(\begin{bmatrix}
0 & \mathbf{X} \\
\mathbf{Y} & 0
\end{bmatrix}\right)&\geq&  \left(\sum_{k=1}^{d}\big|\langle \begin{bmatrix}
0 & X_k \\
Y_k & 0
\end{bmatrix}z,z\rangle\big|^2\right)^\frac12
= \frac{\left(\sum_{k=1}^{d}|\langle X_k y, x\rangle+\langle Y_k x,y\rangle|^2\right)^\frac12}{(\|x\|^2+\|y\|^2)}.
\end{eqnarray*}
Therefore, $\left(\|x\|^2+\|y\|^2\right) w_e\left(\begin{bmatrix}
0 & \mathbf{X} \\
\mathbf{Y} & 0
\end{bmatrix}\right)\geq\left(\sum_{k=1}^{d}|\langle X_k y, x\rangle+\langle Y_k x,y\rangle|^2\right)^\frac12$
for all $x,y\in\mathscr{H}$. This implies that $\left(\|x\|^2+\|y\|^2\right)w_e\left(\begin{bmatrix}
0 & \mathbf{X} \\
\mathbf{Y} & 0
\end{bmatrix}\right)\geq|\langle X_k y, x\rangle+\langle Y_k x,y\rangle|$ holds for each $k=1,2,\ldots,d.$
Now, replacing $x$ and $y$ by $A_kx$ and $B_kx$, respectively, and then summing, we get \begin{eqnarray*}
\sum_{k=1}^{d}\big|\langle X_kB_kx,A_kx\rangle+\langle Y_kA_kx, B_kx\rangle\big|&\leq& \sum_{k=1}^{d}\left(\left(\|A_kx\|^2+\|B_kx\|^2\right)w_e\left(\begin{bmatrix}
0 & \mathbf{X} \\
\mathbf{Y} & 0
\end{bmatrix}\right)\right)\\
&=& w_e\left(\begin{bmatrix}
0 & \mathbf{X} \\
\mathbf{Y} & 0
\end{bmatrix}\right)\sum_{k=1}^{d}\left(\|A_kx\|^2+\|B_kx\|^2\right)\\
&\leq&  w_e\left(\begin{bmatrix}
0 & \mathbf{X} \\
\mathbf{Y} & 0
\end{bmatrix}\right)\left(\|\mathbf{A}\|^2+\|\mathbf{B}\|^2\right)\|x\|^2.
\end{eqnarray*}
So,
\begin{eqnarray*}
\left(\sum_{k=1}^{d}\big|\langle X_kB_kx,A_kx\rangle+\langle Y_kA_kx, B_kx\rangle\big|^2\right)^\frac12 \nonumber &\leq& \sum_{k=1}^{d}\big|\langle X_kB_kx,A_kx\rangle+\langle Y_kA_kx, B_kx\rangle\big|\\
&\leq& w_e\left(\begin{bmatrix}
0 & \mathbf{X} \\
\mathbf{Y} & 0
\end{bmatrix}\right)\left(\|\mathbf{A}\|^2+\|\mathbf{B}\|^2\right)\|x\|^2.
\label{eqn12}\end{eqnarray*} 
Taking supremum over $\|x\|=1$, we have
\begin{eqnarray}
 w_e(\mathbf{A^*XB+B^*YA})\leq\left(\|\mathbf{A}\|^2+\|\mathbf{B}\|^2\right) w_e\left(\begin{bmatrix}
0 & \mathbf{X} \\
\mathbf{Y} & 0
\end{bmatrix}\right).\label{eqn10}\end{eqnarray} 
Now, the desired inequality follows from \eqref{eqn10} by replacing $\mathbf{A}$ and $\mathbf{B}$ by t$\mathbf{A}$ and $\frac{1}{t}$ $\mathbf{B}$, respectively, where $ t=\sqrt{\frac{\|\mathbf{B}\|}{\|\mathbf{A}\|}}.$ In particular, taking $\mathbf{X}=\mathbf{Y}$ we achieve the desired second inequality.
\end{proof}

Finally, we obtain the following norm inequality.

\begin{theorem}
	Let  $\mathbf{X}=(X_1,X_2, \ldots,X_d)$, $\mathbf{Y}=(Y_1,Y_2, \ldots,Y_d)$,  $\mathbf{Z}=(Z_1,Z_2, \ldots,Z_d),$  $\mathbf{W}=(W_1,W_2, \ldots,W_d)\in\mathbb{B}^d(\mathscr{H}).$ Then $$ \norm{\begin{bmatrix}
			\mathbf{X} & \mathbf{Y} \\
			\mathbf{Z} & \mathbf{W} 
	\end{bmatrix}}\leq \norm{\begin{bmatrix}
			\|\mathbf{X}\| & \|\mathbf{Y}\| \\
			\|\mathbf{Z}\| & \|\mathbf{W}\| 
	\end{bmatrix}}.$$  
	\label{them3}\end{theorem}
\begin{proof}
	Let $u=(x,y)\in\mathscr{H}\oplus \mathscr{H}$ with $\|u\|=1,$ i.e., $\|x\|^2+\|y\|^2=1.$ Then we have,
	\begin{eqnarray*}
		&&\sum_{k=1}^{d}\norm{\begin{bmatrix}
				X_k & Y_k \\
				Z_k & W_k
			\end{bmatrix}u}^2\\
		&=&\sum_{k=1}^{d}\langle\begin{bmatrix}
			X_k & Y_k \\
			Z_k & W_k
		\end{bmatrix}u, \begin{bmatrix}
			X_k & Y_k \\
			Z_k & W_k
		\end{bmatrix}u\rangle\\
		&=&\sum_{k=1}^{d}\left(\|X_kx+Y_ky\|^2+\|Z_kx+W_ky\|^2\right)\\
		&=&\sum_{k=1}^{d}( \|X_kx\|^2+\|Y_ky\|^2+\|Z_kx\|^2+\|W_ky\|^2+2 Re\langle X_kx, Y_ky\rangle+2 Re\langle W_ky, Z_kx\rangle)\\
		&&\,\,\,\,\,(\textit{here $Re\langle x,y\rangle$ means real part of $\langle x, y\rangle$})\\
		&\leq&\sum_{k=1}^{d}( \|X_kx\|^2+\|Y_ky\|^2+\|Z_kx\|^2+\|W_ky\|^2+2|\langle X_kx, Y_ky\rangle|+2|\langle W_ky, Z_kx\rangle|)\\
		&\leq&\sum_{k=1}^{d} \|X_kx\|^2+\sum_{k=1}^{d} \|Y_ky\|^2+\sum_{k=1}^{d} \|Z_kx\|^2+\sum_{k=1}^{d} \|W_ky\|^2 \\
		&& +2\sum_{k=1}^{d} \|X_kx\|\|Y_ky\|+2\sum_{k=1}^{d} \|Z_kx\|\|W_ky\|\\
		&\leq& \|\mathbf{X}\|^2\|x\|^2+\|\mathbf{Y}\|^2\|y\|^2+\|\mathbf{Z}\|^2\|x\|^2+\|\mathbf{W}\|^2\|y\|^2\\&&\hspace{1.0cm}+2\left(\sum_{k=1}^{d}\|X_kx\|^2\right)^\frac12\left(\sum_{k=1}^{d}\|Y_ky\|^2\right)^\frac12+2\left(\sum_{k=1}^{d}\|Z_kx\|^2\right)^\frac12\left(\sum_{k=1}^{d}\|W_ky\|^2\right)^\frac12,\\&&\,\,\,(\textit{by Cauchy-Schwarz inequality})\\
		&\leq& \left(\|\mathbf{X}\|^2+\|\mathbf{Z}\|^2\right)\|x\|^2+ \left(\|\mathbf{Y}\|^2+\|\mathbf{W}\|^2\right)\|y\|^2+2\|\mathbf{X}\|\|\mathbf{Y}\|\|x\|\|y\|+2\|\mathbf{Z}\|\|\mathbf{W}\|\|x\|\|y\|\\
		&=& \langle \begin{bmatrix}
			\|\mathbf{X}\| & \|\mathbf{Y}\| \\
			\|\mathbf{Z}\| & \|\mathbf{W}\| 
		\end{bmatrix}^*\begin{bmatrix}
			\|\mathbf{X}\| & \|\mathbf{Y}\| \\
			\|\mathbf{Z}\| & \|\mathbf{W}\| 
		\end{bmatrix} \widetilde{x},\widetilde{x}\rangle\,\,\,\textit{(here $\widetilde{x}=(
			\|x\|,
			\|y\|)\in \mathbb{C}^2$})\\
		&\leq& \norm{\begin{bmatrix}
				\|\mathbf{X}\| & \|\mathbf{Y}\| \\
				\|\mathbf{Z}\| & \|\mathbf{W}\| 
			\end{bmatrix}^*\begin{bmatrix}
				\|\mathbf{X}\| & \|\mathbf{Y}\| \\
				\|\mathbf{Z}\| & \|\mathbf{W}\| 
		\end{bmatrix}}\\
		&=&\norm{\begin{bmatrix}
				\|\mathbf{X}\| & \|\mathbf{Y}\| \\
				\|\mathbf{Z}\| & \|\mathbf{W}\| 
		\end{bmatrix}}^2.
 	\end{eqnarray*}
	Therefore, the desired inequality follows by taking supremum over $\|u\|=1$.
\end{proof}

\noindent 
{\textbf{Statements and Declarations}}\\
Data sharing not applicable to this article as no datasets were generated or analysed during the current study.
Authors also declare that there is no financial or non-financial interests that are directly or indirectly related to the work submitted for publication.  On behalf of all authors, the corresponding author states that there is no conflict of interest.

\bibliographystyle{amsplain}

\end{document}